\documentclass[10pt]{amsart}
\usepackage{amssymb}
\usepackage{xcolor}

\newtheorem{proposition}{Proposition}[section]
\newtheorem{lemma}[proposition]{Lemma}
\newtheorem{corollary}[proposition]{Corollary}
\newtheorem{theorem}[proposition]{Theorem}

\theoremstyle{definition}
\newtheorem{definition}[proposition]{Definition}

\newcommand{\selabel}[1]{\label{se:#1}}

\def\<{\leqslant}
\def\>{\geqslant}

\date{}

\begin{document}

\title{THE RIBBON ELEMENTS OF DRINFELD DOUBLE OF RADFORD HOPF ALGEBRA}

\author{Hua Sun}
\address{Hua Sun \\ College of Mathematical Science \\ Yangzhou University \\ Yangzhou 225002, China}
\email{huasun@yzu.edu.cn}
\author{Yuyan Zhang}
\address{Yuyan Zhang \\ College of Mathematical Science \\ Yangzhou University \\ Yangzhou 225002, China}
\email{2779082883@qq.com}
\author{Libin Li}
\address{Libin Li \\ College of Mathematical Science \\ Yangzhou University \\
	Yangzhou 225002, China}
\email{lbli@yzu.edu.cn}
\subjclass[2020]{16T05}
\keywords{Drinfeld double, ribbon Hopf algebra, quasi-triangular Hopf algebra}
\thanks{This work was financially supported by NNSF of China (Nos. 12371041, 12201545) and Natural Science Foundation of the Jiangsu Higher Education Institutions of China(No. 22KJD110006)}
\begin{abstract}
Let $m$, $n$ be two positive integers,  $\Bbbk$ be an algebraically closed field with char($\Bbbk)\nmid mn$. {Radford constructed an $mn^{2}$-dimensional Hopf algebra} {$R_{mn}(q)$ such that its Jacobson radical is not a Hopf ideal.} We show that the Drinfeld double $D(R_{mn}(q))$ of Radford Hopf algebra $R_{mn}(q)$ has ribbon elements if and only if $n$ is odd. Moreover, if $m$ is even and $n$ is odd, then $D(R_{mn}(q))$ {has} two ribbon elements,  if both $m$ and $n$ are odd, then $D(R_{mn}(q))$ has only one ribbon element. Finally, we compute explicitly all ribbon elements of $D(R_{mn}(q))$.
\end{abstract}
\maketitle

\section{\bf Introduction}\selabel{1}
The representation category of a Hopf algebra is a tensor category. Moreover, the representation category of a quasi-triangular Hopf algebra is a braided tensor category. The braiding structure of a quasi-triangular Hopf algebra can supply solution to the Yang-Baxter equation. {Many results about quasi-triangular Hopf algebra} {can be found in \cite{CQ,WD}.} Drinfeld \cite{VG} gave a general method, by which one can construct a quasi-triangular Hopf algebra from a finite dimensional Hopf algebra, i.e. the quantum double(or Drinfeld double) of a Hopf algebra. {A quasi-triangular Hopf} {algebra which has a ribbon element is called ribbon Hopf algebra.} {Finite dimensional} {ribbon Hopf algebras play} an important role in constructing invariants of 3-manifolds \cite{Re}. Thus, researchers {have} paid much attention to the question that when a quasi-triangular Hopf algebra has ribbon structures. Kauffman and Radford \cite{Kau} gave a necessary and sufficient condition for the Drinfeld double of a finite dimensional Hopf algebra {to} admit {a} ribbon structure, they proved that $(D(A_n(q)),\mathcal{R})$ is a ribbon Hopf algebra if and only if $n$ is odd, where $D(A_n(q))$ is the Drinfeld double of $n^2$-dimensional Taft algebra $A_n(q)$, $\mathcal{R}$ is the universal $\mathcal{R}$-matrix of $D(A_n(q))$. In \cite{2022}, Benkart and Biswal computed the ribbon element of $(D(A_n(q)),\mathcal{R})$ explicitly when $n$ is odd. In \cite{YSL}, Chen and Yang gave a necessary and sufficient condition for the Drinfeld double of a finite dimensional Hopf superalgebra to have a ribbon element. Burciu \cite{Bur} gave a sufficient condition for the Drinfeld double $(D(u(\mathcal{D},0,0)),\mathcal{R})$ to be a ribbon Hopf algebra, where $u(\mathcal{D},0,0)$ was {constructed} by Andruskiewitsch and Schneider in \cite{Andru}. {Leduc and Ram\cite{Led} showed how the ribbon Hopf algebra} {structure on the Drinfeld-Jimbo quantum groups of types $A$, $B$, $C$, and $D$ can be} {used to derive formulas giving explicit realizations of the irreducible representations} {of the Iwahori-Hecke algebras of type $A$ and the Birman-Wenzl algebras. 

{Radford\cite{DE} constructed an $mn^{2}$-dimensional Hopf algebra such that its Jacobson } {radical is not a Hopf ideal. The Hopf algebra is denoted by $R_{mn}(q)$ and called a } {Radford Hopf algebra here.} In this paper, we give a sufficient and necessary condition for the Drinfeld double $D(R_{mn}(q))$ of $R_{mn}(q)$ to have ribbon structure. The paper is organized as follows. In Section 2, we recall some definitions, {notions} and the structures of Radford Hopf algebra $R_{mn}(q)$. In Section 3, we describe the Hopf algebra structure of  $(R_{mn}(q))^{*}$. In Section 4, we show that $(D(R_{mn}(q)),\mathcal{R})$ {has} ribbon elements if and only if $n$ is odd. Moreover, we compute all ribbon elements of $(D(R_{mn}(q)),\mathcal{R})$.

\numberwithin{equation}{section}
\section{\bf Preliminaries}\selabel{2}

Throughout, we work over an algebraically closed field $\Bbbk$ with char($\Bbbk)\nmid mn$. Unless
otherwise stated, all algebras and Hopf algebras are
defined over $\Bbbk$;
$dim$ and $\otimes$ denote $dim_{\Bbbk}$ and $\otimes_\Bbbk$ respectively. The references \cite{Ka, Mon, Sw} are basic references for the theory of Hopf algebras and quantum groups.

Let $0\neq q\in \Bbbk.$ For any nonnegative integer $n$, define $(n)_{q}$ by $(0)_{q}=0$ and $(n)_{q}=1+q+...+q^{n-1}$ for $n>0$. Observe that $(n)_{q}=n$ when $q=1$, and
$$(n)_{q}=\frac{q^{n}-1}{q-1}$$
when $q\neq1$. Define the $q$-factorial of $n$ by
$$(n)!_{q}=\frac{(q^{n}-1)(q^{n-1}-1)...(q-1)}{(q-1)^{n}}$$
when $n>0$ and $q\neq1$. The $q$-binomial coefficients $\left(
\begin{array}{c}
n\\
i
\end{array}
\right)_{q}$ are defined inductively as follows for $0\leq i\leq n$:
$$\left(
\begin{array}{c}
n\\
0
\end{array}
\right)_{q}=1=\left(
\begin{array}{c}
n\\
n
\end{array}
\right)_{q},
\ for\ n\geq0,$$
$$\left(
\begin{array}{c}
n\\
i
\end{array}
\right)_{q}=q^{i}\left(
\begin{array}{c}
n-1\\
i
\end{array}
\right)_{q}+\left(
\begin{array}{c}
n-1\\
i-1
\end{array}
\right)_{q},\ for\ 0<i<n. $$
It is well-known that
$\left(\begin{array}{c}
n\\
i
\end{array}
\right)_{q}$ is a polynomial in $q$ with integer coefficients and with value at $q=1$ equal to the usual binomial coefficient $\left(\begin{array}{c}
n\\
i
\end{array}
\right)$ , and that $$\left(\begin{array}{c}
n\\
i
\end{array}
\right)_{q}=\frac{(n)!_{q}}{(i)!_{q}(n-i)!_{q}}$$
when $(n-1)!_{q}\neq0$ and $0<i<n$.

In the next, we use the {Sweedler's} notation: for $x\in H$, $H$ being a coalgebra,
$$\Delta(x)=\sum x_{1}\otimes x_{2}.$$
{For $h\in H$ and $\alpha$ in the dual space $H^{*}$, we define}
$$\langle\alpha,h\rangle :=\alpha(h)\in \Bbbk.$$
Suppose that $H$ is a bialgebra over $\Bbbk$. The left and right $H$-module actions defined on $H^{*}$ by
$$<a\rightharpoonup p,b>=<p,ba>=<p\leftharpoonup b,a>$$
respectively for $a,b\in H$ and $p\in H^{*}$ give $H^{*}$ an {$H$}-bimodule structure. Likewise the left and a right $H^{*}$-{module} actions on $H$ by
$$p\rightharpoonup a=\sum a_{1}<p,a_{2}>,\ a\leftharpoonup p=\sum<p,a_{1}>a_{2}$$
respectively for $p\in H^{*}$ and $a\in H$ give $H$ a $H^{*}$-bimodule structure.
\subsection{Radford Hopf algebra}\selabel{2.1}
~~

In this subsection, we recall the structure of Radford Hopf algebra $R_{mn}(q)$.

Let $m\geq2$, $n\geq1$, and let $q\in \Bbbk$ be a primitive $n$-th root of unity. The Radford Hopf algebra $R_{mn}(q)$ is generated as an algebra by $g$ and $x$ subject to the following relations:
$$g^{mn}=1,\ x^{n}=g^{n}-1,\ xg=qgx.$$
The comultiplication $\Delta$, counit $\varepsilon$ and antipode $S$ are given respectively by
$$\Delta(x)=x\otimes g+1\otimes x,\ \varepsilon(x)=0,\ S(x)=-xg^{-1},$$
$$\Delta(g)=g\otimes g,\ \varepsilon(g)=1, \ S(g)=g^{-1}=g^{mn-1}.$$
Note that dim$(R_{mn}(q))=mn^{2}$, and $R_{mn}(q)$ has a $\Bbbk$-basis $\{g^{i}x^{j}|0\leq i\leq mn-1,0\leq j\leq n-1\}$. For details, one can refer to \cite{DE}.
\subsection{Ribbon Hopf algebra}\selabel{2.2}
~~

In this subsection, we recall the definition of quasi-triangular Hopf algebra and ribbon Hopf algebra.
\begin{definition}\label{3.1} Let $H$ be a Hopf algebra. If there exists an invertible element $\mathcal{R}\in H\otimes H$, such that
\begin{subequations}
\begin{align*}
&\mathcal{R}\Delta(x)=\Delta^{op}(x)\mathcal{R},\ for\ all\ x\in H,\\
&(\Delta\otimes id)\mathcal{R}=\mathcal{R}_{13}\mathcal{R}_{23},\\
&(id\otimes\Delta)\mathcal{R}=\mathcal{R}_{13}\mathcal{R}_{12},
\end{align*}
\end{subequations}
then $H$ is called a quasi-triangular Hopf algebra. Here $\Delta^{op}(x)$ has the tensor factors in $\Delta(x)$ interchanged, and $\mathcal{R}=\sum_{i}x_{i}\otimes y_{i},\mathcal{R}_{12}=\sum_{i}x_{i}\otimes y_{i}\otimes1,\mathcal{R}_{13}=\sum_{i}x_{i}\otimes1\otimes y_{i},\mathcal{R}_{23}=\sum_{i}1\otimes x_{i}\otimes y_{i}.$ Let $\mathcal{R}^{op}=\sum_{i}y_{i}\otimes x_{i}.$
\end{definition}
We assume $\mathcal{R}=\sum_{i}x_{i}\otimes y_{i}$ as above and use the antipode $S$ to define
\begin{equation}\label{2}
u=\sum_{i}S(y_{i})x_{i}\in H.
\end{equation}
Then the following expressions hold
$$uxu^{-1}=S^{2}(x) \ for \ all\ x\in H \ and \ \Delta(u)=(\mathcal{R}^{op}\mathcal{R})^{-1}(u\otimes u).$$
\begin{definition}\label{2.1}
Let $H$ be a quasi-triangular Hopf algebra. If there exists an invertible element $v$(the ribbon element) in the center of $H$ such that
\begin{equation}\label{2.2}
v^{2}=uS(u),\ S(v)=v,\ \varepsilon(v)=1,\ \Delta(v)=(\mathcal{R}^{op}\mathcal{R})^{-1}(v\otimes v),
\end{equation}
where $u$ is as in (\ref{2}), then $(H,\mathcal{R},v)$ is called a ribbon Hopf algebra.
\end{definition}

\section{\bf The structure of $(R_{mn}(q))^{*}$}\selabel{3}
~~
In this section, we describe the Hopf algebra structure of $(R_{mn}(q))^{*}$.

Let $\{\overline{g^{i}x^{j}}|0\leq i\leq mn-1,0\leq j\leq n-1\}$be the basis of Hopf algebra $(R_{mn}(q))^{*}$ such that $\overline{g^{i}x^{j}}(g^{i}x^{j})=1$ and $\overline{g^{i}x^{j}}(g^{i^{'}}x^{j^{'}})=0$ for $(i^{'},j^{'})\neq (i,j),\ 0\leq i,i^{'}\leq mn-1,\ 0\leq j,j^{'}\leq n-1.$
\begin{lemma}\label{L3.1} Let $0\leq i,k\leq mn-1$ and $0\leq j,l\leq n-1$. Then
$$\overline{g^{i}x^{j}}\ast\overline{g^{k}x^{l}}=
\begin{cases}
\  \ \ \   \   0,& \text{if $k\neq i+j$(mod $mn$) or $l+j\geq n,$}\\
\left(
\begin{array}{c}
l+j\\
j
\end{array}
\right)_{q}\overline{g^{i}x^{j+l}},& \text{otherwise.}
\end{cases}$$
\end{lemma}
\begin{proof}It follows from a straightforward verification.
\end{proof}
Obviously, $\sum_{i=0}^{mn-1}\overline{g^{i}}=\varepsilon$, which is the identity of the algebra $(R_{mn}(q))^{*}$.

Let $\xi\in \Bbbk$ be a primitive $mn$-th root of unity with $\xi^{m}=q$. Put $\alpha=\sum^{mn-1}_{i=0}\xi^{i}\overline{g^{i}}$ and $\beta=\sum^{mn-1}_{i=0}\overline{g^{i}x}$.

\begin{lemma}\label{L3.2} $(R_{mn}(q))^{*}$ is generated as an algebra by $\alpha$ and $\beta$.
\end{lemma}
\begin{proof}
Obviously, $\alpha,\beta\in(R_{mn}(q))^{*}.$ Let $A$ be a subalgebra of $(R_{mn}(q))^{*}$ generated by $\alpha$ and $\beta$. It follows from Lemma \ref{L3.1} that $\beta^{j}=(j)!_{q}(\overline{x^{j}}+\overline{gx^{j}}+...+\overline{g^{mn-1}x^{j}}),1\leq j\leq n-1$, and
\begin{equation}\label{gs3.1}
\begin{aligned}
\alpha&=\overline{1}+\xi\overline{g}+\xi^{2}\overline{g^{2}}+...+\xi^{mn-1}\overline{g^{mn-1}},\\
\alpha^{2}&=\overline{1}+\xi^{2}\overline{g}+\xi^{4}\overline{g^{2}}+...+\xi^{2(mn-1)}\overline{g^{mn-1}},\\
&...\\
\alpha^{mn-1}&=\overline{1}+\xi^{mn-1}\overline{g}+\xi^{2(mn-1)}\overline{g^{2}}+...+\xi^{(mn-1)(mn-1)}\overline{g^{mn-1}}.
\end{aligned}
\end{equation}
Then we have
\begin{equation*}
\begin{aligned}
\frac{1}{(j)!_{q}}\alpha\beta^{j}&=\overline{x^{j}}+\xi\overline{gx^{j}}+\xi^{2}\overline{g^{2}x^{j}}+...+\xi^{mn-1}\overline{g^{mn-1}x^{j}},\\
\frac{1}{(j)!_{q}}\alpha^{2}\beta^{j}&=\overline{x^{j}}+\xi^{2}\overline{gx^{j}}+\xi^{4}\overline{g^{2}x^{j}}+...+\xi^{2 (mn-1)}\overline{g^{2(mn-1)}x^{j}},\\
&...\\
\frac{1}{(j)!_{q}}\alpha^{mn-1}\beta^{j}&=\overline{x^{j}}+\xi^{mn-1}\overline{gx^{j}}+\xi^{2(mn-1)}\overline{g^{2}x^{j}}+...+\xi^{(mn-1)(mn-1)}\overline{g^{mn-1}x^{j}},\\
\frac{1}{(j)!_{q}}\alpha^{mn}\beta^{j}&=\overline{x^{j}}+\overline{gx^{j}}+\overline{g^{2}x^{j}}+...+\overline{g^{mn-1}x^{j}}.
\end{aligned}
\end{equation*}

The coefficient determinant of (\ref{gs3.1}) is

$|B|=$
$
\left |
\begin{matrix}
1 & \xi & \xi^{2}&\cdots& \xi^{mn-1}\\
1 & \xi^{2} & \xi^{4}&\cdots& \xi^{2(mn-1)}\\
                &\cdots\\
1 & \xi^{mn-1} & \xi^{2(mn-1)}&\cdots& \xi^{(mn-1)(mn-1)}\\
\end{matrix}
\right |
$
$=\prod_{0\leq i<j\leq mn-1}(\xi^{i}-\xi^{j})\neq0.$

Therefore, by Cramer's Rule, we have

$$\overline{g}=\frac{
\left |
\begin{matrix}
1 & \alpha & \xi^{2}&\cdots& \xi^{mn-1}\\
1 & \alpha^{2} & \xi^{4}&\cdots& \xi^{2(mn-1)}\\
                &\cdots\\
1 & \alpha^{mn-1} & \xi^{2(mn-1)}&\cdots& \xi^{(mn-1)(mn-1)}\\
\end{matrix}
\right |
}{
\left |
\begin{matrix}
1 & \xi & \xi^{2}&\cdots& \xi^{mn-1}\\
1 & \xi^{2} & \xi^{4}&\cdots& \xi^{2(mn-1)}\\
                &\cdots\\
1 & \xi^{mn-1} & \xi^{2(mn-1)}&\cdots& \xi^{(mn-1)(mn-1)}\\
\end{matrix}
\right |
}\in A.$$

Similarly, one can prove that $\overline{g^{i}x^{j}}\in A$ for $0\leq i\leq mn-1, 0\leq j \leq n-1.$ The lemma is proved.

\end{proof}

\begin{proposition}\label{2.15} In $(R_{mn}(q))^{*}$, we have
$$\alpha^{mn}=\varepsilon,\ \beta^{n}=0,\ \beta\alpha=\xi\alpha\beta.$$
\end{proposition}
\begin{proof}It follows from Lemma \ref{L3.1}.
\end{proof}
For any positive integer $s$, let $s^{\dagger}$ denote the unique integer such that $0\leq s^{\dagger}\leq mn-1$ and $mn\mid (s-s^{\dagger})$. Then we have the following proposition.
\begin{proposition}\label{2.12} The comultiplication, the counit and the antipode of $((R_{mn}(q))^{*})^{cop}$ are given by
$$\Delta(\alpha)=\alpha\otimes\alpha+(\xi^{n}-1)\sum_{\substack{k+l=n\\0<k<n}}\frac{1}{(k)!_{q}(l)!_{q}}\alpha^{mk+1}\beta^{l}\otimes\alpha\beta^{k},$$
$$\Delta(\beta)=\beta\otimes1+\alpha^{m}\otimes\beta,$$
$$\varepsilon(\alpha)=1,\ \varepsilon(\beta)=0,$$
$$S(\alpha)=\alpha^{mn-1},\ S(\beta)=-\alpha^{-m}\beta.$$
\end{proposition}
\begin{proof}We only consider the formula of $\Delta(\alpha)$ since the proof for $\Delta(\beta)$ is similar. In $(R_{mn}(q))^{*}$, we have
\begin{equation*}
\begin{split}
\Delta(\overline{g^{t}})&=\sum_{(i+j)^{\dagger}=t}\overline{g^{i}}\otimes\overline{g^{j}}-\sum_{\substack{(k+l)=n\\0<k<n}}\sum_{(i+j)^{\dagger}=t}q^{jk}\overline{g^{i}x^{k}}\otimes\overline{g^{j}x^{l}}\\
&\  \ \
+\sum_{\substack{(k+l)=n\\0<k<n}}\sum_{(i+j)^{\dagger}=(m-1)n+t}q^{jk}\overline{g^{i}x^{k}}\otimes\overline{g^{j}x^{l}},
\end{split}
\end{equation*}
where $(i+j)^{\dagger}$ is defined as above. Moreover,
$$\Delta(\alpha)=\alpha\otimes\alpha+(\xi^{n}-1)\sum_{\substack{k+l=n\\0<k<n}}\sum_{(i+j)^{\dagger}=0}(\xi^{n}-1)q^{jk}\frac{\xi^{-i}}{(k)!_{q}}\alpha\beta^{k}\otimes\overline{g^{j}x^{l}}.$$
Hence,
$$\Delta(\alpha)=\alpha\otimes\alpha+(\xi^{n}-1)\sum_{\substack{k+l=n\\0<k<n}}\frac{1}{(k)!_{q}(l)!_{q}}\alpha\beta^{k}\otimes\alpha^{mk+1}\beta^{l}.$$
Consequently, in $((R_{mn}(q))^{*})^{cop}$, we have
$$\Delta(\alpha)=\alpha\otimes\alpha+(\xi^{n}-1)\sum_{\substack{k+l=n\\0<k<n}}\frac{1}{(k)!_{q}(l)!_{q}}\alpha^{mk+1}\beta^{l}\otimes\alpha\beta^{k}.$$
Similarly, one can show the formulas of the counit $\varepsilon$ and antipode $S$ of $((R_{mn}(q))^{*})^{cop}$.
\end{proof}
By \cite[IX.4]{Ka}, we have the following propositions:
\begin{proposition}\label{2.12}\cite[Definition IX.4.1.]{Ka} The quantum double $D(R_{mn}(q))$ of the Hopf algebra $R_{mn}(q)$ is the bicrossed product of $R_{mn}(q)$ and of $(R_{mn}(q)^{*})^{cop}.$
$$D(R_{mn}(q))=(R_{mn}(q)^{*})^{cop}\bowtie R_{mn}(q).$$
\end{proposition}
By Proposition \ref{2.12} and Lemma \ref{L3.2}, one knows that $D(R_{mn}(q))$ is generated as an algebra by $\varepsilon\bowtie g$, $\varepsilon\bowtie x$, $\alpha\bowtie1$, $\beta\bowtie1$.
\begin{proposition}\label{3.6}\cite[Lemma IX.4.2.]{Ka} The multiplication, comultiplication and counit in $D(R_{mn}(q))$ are given by
$$(f\bowtie a)(g\bowtie b)=\sum_{(a)}f(a_{1}\rightharpoonup g\leftharpoonup s^{-1}(a_{3}))\bowtie a_{2}b,$$
$$\varepsilon(f\bowtie a)=\varepsilon(a)f(1),$$
$$\Delta(f\bowtie a)=\sum_{(a)(f)}(f_{1}\bowtie a_{1})(f_{2}\bowtie a_{2}),$$
\begin{equation*}
\begin{split}
S(f\bowtie a)&=\sum(S(a_{2})\rightharpoonup S(f_{1}))\bowtie (f_{2}\rightharpoonup S(a_{1}))\\
&=\sum(S(f_{2})\leftharpoonup a_{1})\bowtie (S(a_{2})\leftharpoonup S(f_{1})),
\end{split}
\end{equation*}
where $f,g\in (R_{mn}(q)^{*})^{cop}$ and $a,b\in R_{mn}(q),\ \sum f_{1}(x)f_{2}(y)=f(yx)$ for all $x,y\in R_{mn}(q)$.
\end{proposition}

\section{\bf The ribbon elements of $(D(R_{mn}(q)),\mathcal{R})$}\selabel{4}
In this section, we recall some results about quasi-triangular Hopf algebras and use them to determine the existence of ribbon elements of $(D(R_{mn}(q)),\mathcal{R})$ and {give} explicit expressions for these ribbon elements.
\subsection{\bf Universal R-matrix of $D(R_{mn}(q))$}\selabel{s4.1}
~~

In this subsection, we determine the universal R-matrix of $D(R_{mn}(q))$.

By \cite[IX.4.2.]{Ka}, the universal $\mathcal{R}$-matrix of the quantum double has {an} explicit formula:
$$\mathcal{R}=\sum_{i\in I}(1\bowtie e_{i})\otimes (e^{i}\bowtie 1),$$
where $\{e_{i}\}_{i\in I}$ is a basis of the vector space $H$ and $\{e^{i}\}_{i\in I}$ is its dual basis in $(H^{op})^{*}=(H^{*})^{cop}$.
\begin{lemma}\label{L4.1}For any $0\leq i\leq mn-1,\ 0\leq j\leq n-1$, set
$$y_{i,j}=\frac{1}{mn}\frac{1}{(j)!_{q}}\sum^{mn-1}_{k=0}\xi^{-ik}\alpha^{k}\beta^{j}$$
in $(R_{mn}(q))^{*}$. Then $y_{i,j}(g^{i_{1}}x^{j_{1}})=\delta_{i,i_{1}}\delta_{j,j_{1}}$ for all
  $0\leq i,i_{1}\leq mn-1,\ 0\leq j,j_{1}\leq n-1$.
\end{lemma}

\begin{proof} Let $0\leq i\leq mn-1$, $0\leq j\leq n-1$, $\delta_{i,j}$ {be the Kronecker} symbol, then we have
\begin{equation*}
\begin{split}
&y_{i,j}(g^{i_{1}}x^{j_{1}})\\
&=\frac{1}{mn}\frac{1}{(j)!_{q}}(\sum^{mn-1}_{k=0}\xi^{-ik}\alpha^{k}\beta^{j})(g^{i_{1}}x^{j_{1}})\\
&=\frac{1}{mn}\frac{1}{(j)!_{q}}(j)!_{q}(\beta^{j}+\xi^{-1}\alpha\beta^{j}+...+\xi^{-(mn-1)i}\alpha^{mn-1}\beta^{j})(g^{i_{1}}x^{j_{1}})\\
&=\frac{1}{mn}(\sum_{v=0}^{mn-1}\overline{g^{v}x^{j}}+\xi^{-i}\sum_{v=0}^{mn-1}\xi^{v}\overline{g^{v}x^{j}}+...+\xi^{-(mn-1)i}\sum_{v=0}^{mn-1}\xi^{(mn-1)v}\overline{g^{v}x^{j}})(g^{i_{1}}x^{j_{1}})\\
&=\frac{1}{mn}(1+\xi^{(i_{1}-i)}+...+\xi^{(i_{1}-i)(mn-1)})\delta_{j,j_{1}}\\
&=\delta_{i,i_{1}}\delta_{j,j_{1}}.\\
\end{split}
\end{equation*}
\end{proof}

By Lemma \ref{L4.1}, one can easily know the $\mathcal{R}$-matrix of $D(R_{mn}(q))$ is
\begin{equation*}
\begin{split}
\mathcal{R}&=\sum_{i,j}(1\bowtie g^{i}x^{j})\otimes(y_{i,j}\bowtie 1)\\
&=\frac{1}{mn}\sum_{i,j,k}\frac{1}{(j)!_{q}}\xi^{-ik}(1\bowtie g^{i}x^{j})\otimes(\alpha^{k}\beta^{j}\bowtie 1).\\
\end{split}
\end{equation*}

\subsection{The existence of ribbon elements}\selabel{4.1}
~~

In this subsection, we review some facts about integral and quasi-ribbon element for a finite-dimensional Hopf algebra $H$.
The following results on integral can be found in \cite[Chapter 2]{Mon}:

$\bullet$ A left integral element in $H$ is an element $t$ in $H$ such that $ht=\varepsilon(h)t,\forall h\in H.$ A right integral element in $H$ is an element $t^{'}$ in $H$ such that $t^{'}h=\varepsilon(h)t^{'},\forall h\in H.$

$\int^{l}_{H}$ denotes the subspace of left integrals in $H$ and $\int^{r}_{H}$ denotes the subspace of right integrals in $H$. $H$ is called \emph{unimodular} if $\int^{l}_{H}=\int^{r}_{H}$.

$\bullet$ Let $H$ be a finite-dimensional Hopf algebra. Then\\
(1) $\int^{l}_{H}$ and $\int^{r}_{H}$ are {both} one-dimensional. \\
(2) The antipode $S$ of $H$ is bijective and $S(\int^{l}_{H})=\int^{r}_{H},\ S(\int^{r}_{H})=\int^{l}_{H}$.

$\bullet$ Suppose $t\in\int^{l}_{H}$ and $T\in\int^{r}_{H^{*}}$. Notice that left integrals for $H$ form a one-dimensional ideal of $H$. Hence there is a unique $\tilde{\alpha}\in G(H^{*})$ such that $th=\langle\tilde{\alpha},h\rangle t$ for all $h\in H$. The condition that $H$ is unimodular is equivalent to $\tilde{\alpha}=\varepsilon$.

Likewise there is a unique $\tilde{g}\in H$ such that $pT=\langle p,\tilde{g}\rangle T$, for all $p\in H^{*}$. We call $\tilde{\alpha}$ and $\tilde{g}$ be the \emph{distinguished grouplike element} of $H^{*}$ and $H$ respectively.

As above, assume the $R$-matrix is $\mathcal{R}=\sum_{i}x_{i}\otimes y_{i}$ , and define
\begin{equation}\label{4.1}
\begin{split}
g_{\tilde{\alpha}}=\sum_{i}x_{i}\tilde{\alpha}(y_{i}),\ and\ h_{\tilde{\alpha}}=g_{\tilde{\alpha}}\tilde{g}^{-1},
\end{split}
\end{equation}
where $\tilde{\alpha}$ is the distinguished grouplike element of $H^{*}$, and $\tilde{g}$ is the distinguished grouplike element of $H$.

A \emph{quasi-ribbon} element of the Hopf algebra $H$ is an element satisfying all the ribbon conditions in \eqref{2.2} except for the requirement that it is central. {Our approach} {to find an explicit formula for ribbon elements of $D(R_{mn}(q))$ is to use the following} { results about the quasi-ribbon elements.}

\begin{theorem}\label{4.5}\cite[Theorem 1]{Kau}  Let $(H,\mathcal{R})$ be a finite-dimensional quasi-triangular Hopf algebra over a field $\Bbbk$. Suppose $h_{\tilde{\alpha}}^{'}$ is any element of $H$ such that $(h_{\tilde{\alpha}}^{'})^{2}=h_{\tilde{\alpha}}$, $\emph{i.e.}\  h_{\tilde{\alpha}}^{'}$ is any square root of the element $h_{\tilde{\alpha}}$ in (\ref{4.1}). Then $v=uh_{\tilde{\alpha}}^{'}$ is a quasi-ribbon element, where $u$ is as in (\ref{2}).
 Moreover, $v=uh_{\tilde{\alpha}}^{'}$ is a ribbon element of $(H,\mathcal{R})$ if and only if $S^{2}(a)=(h_{\tilde{\alpha}}^{'})^{-1}ah_{\tilde{\alpha}}^{'}$ for all $a\in H$.
\end{theorem}

\begin{theorem}\label{4.6}\cite[Theorem 3]{Kau}  Suppose that $H$ is a finite-dimensional Hopf algebra with antipode $S$ over a {field} $\Bbbk$. Let $\tilde{g}$ and $\tilde{\alpha}$ be the distinguished grouplike elements of $H$ and $H^{*}$, respectively. Then:

(1) $(D(H),\mathcal{R})$ has a quasi-ribbon element if and only if there are $h\in G(H)$ and $\gamma\in G(H^{*})$ such that $h^{2}=\tilde{g}$ and $\gamma^{2}=\tilde{\alpha}$.

(2) $(D(H),\mathcal{R})$ has a ribbon element if and only if there are $h\in G(H)$ and $\gamma\in G(H^{*})$ as in part (1) such that
$$S^{2}(y)=h(\gamma\rightharpoonup y\leftharpoonup\gamma^{-1})h^{-1}$$
for all $y\in H$.
\end{theorem}
\begin{corollary}\label{4.7}\cite[Proposition 3]{Kau} Let $\mathcal{S}$ be the antipode of $D(H)$ and set ${u}=\sum\mathcal{S}(\emph{y}_{i})\emph{x}_{i}$. Then
$$(\gamma,h)\mapsto {u}(\gamma^{-1}\otimes h^{-1})$$
defines a one-one correspondence between those pairs $(\gamma,h)\in G(H^{*})\times G(H)$ such that $h^{2}=\tilde{g}$ and $\gamma^{2}=\tilde{\alpha}$ and the quasi-ribbon elements of $(D(H),\mathcal{R}).$ The ribbon elements correspond to those pairs $(\gamma,h)$ which further satisfy $S^{2}(y)=h(\gamma\rightharpoonup y\leftharpoonup\gamma^{-1})h^{-1}$, for all $y\in H$.
\end{corollary}

By \cite{DE1}, $D(R_{mn}(q))$ is unimodular, so that ${\alpha_{D}}=\varepsilon_{D}$, where ${\alpha_{D}}$ and $\varepsilon_{D}$ are the distinguished grouplike element and the counit of $(D(R_{mn}(q)))^{*}$, respectively.

\begin{lemma}\label{4.16}

(1) $\sum_{i=0}^{mn-1}\xi^{(n-1)i}\alpha^{i}\beta^{n-1}$ is a right integral in $(R_{mn}(q))^{*}$, and the distinguished grouplike element of $R_{mn}(q)$ is $g^{1-n}$.

(2) $\sum_{i=0}^{mn-1} g^{i}x^{n-1}$ is a left integral in $R_{mn}(q)$, and the distinguished grouplike element of $(R_{mn}(q))^{*}$ is $\alpha^{-m}$.
\end{lemma}
\begin{proof} Let $\lambda=\sum_{\substack{0\leq i\leq mn-1\\0\leq j\leq n-1}} a_{ij}\alpha^{i}\beta^{j}.$ We use the definition of a right integral of $(R_{mn}(q))^{*}$:$$\lambda\in\int_{(R_{mn}(q))^{*}}^{r}\Leftrightarrow \lambda h^{*}=\varepsilon(h^{*})\lambda, $$for all $h^{*}\in(R_{mn}(q))^{*}.$
Let $h^{*}=\beta$, we have
$$\sum_{\substack{0\leq i\leq mn-1\\0\leq j\leq n-1}} a_{ij}\alpha^{i}\beta^{j+1}=0.$$
Since $\{\alpha^{i}\beta^{k}|0\leq i\leq mn-1, 0\leq j\leq n-1\}$ is {a} basis of $(R_{mn}(q))^{*}$, and $\beta^{n}=0$, we get
$$\lambda=\sum_{i=0}^{mn-1} a_{i,n-1}\alpha^{i}\beta^{n-1}.$$
Let $h^{*}=\alpha$, we have
$$\sum_{i=0}^{mn-1} a_{i,n-1}\xi^{n-1}\alpha^{i+1}\beta^{n-1}=\sum_{i=0}^{mn-1} a_{i,n-1}\alpha^{i}\beta^{n-1},$$
then $a_{i,n-1}\xi^{n-1}=a_{i+1,n-1}$.

Therefore, $$\lambda=\sum_{i=0}^{mn-1}\xi^{(n-1)i}\alpha^{i}\beta^{n-1}.$$
Let $\tilde{g}=\sum_{\substack{0\leq k\leq mn-1\\0\leq l\leq n-1}} b_{kl}g^{k}x^{l}$ be the distinguished grouplike element of  $R_{mn}(q)$. Then we have $$h^{*}\lambda=h^{*}(\tilde{g})\lambda.$$
Let $h^{*}=\alpha^{j}$, we have
$$\sum_{\substack{0\leq i\leq mn-1\\0\leq j\leq n-1}}\xi^{(n-1)i}\alpha^{i+j}\beta^{n-1}=\langle\sum_{\substack{0\leq i\leq mn-1\\0\leq j\leq n-1}}\xi^{ij}\overline{g^{i}},\sum_{\substack{0\leq k\leq mn-1\\0\leq l\leq n-1}} b_{kl}g^{k}x^{l}\rangle\sum_{i=0}^{mn-1}\xi^{(n-1)i}\alpha^{i}\beta^{n-1}.$$
Then $\sum_{k=0}^{mn-1}\xi^{kj}b_{k0}=\xi^{(1-n)j},\ 0\leq j\leq mn-1.$
By Cramer's Rule and Vandermonde determinant, we have
$$b_{k0}=
\begin{cases}
1,& \text{$k=mn+1-n,$}\\
0,& \text{otherwise.}
\end{cases},\ where\ 0\leq k\leq mn-1.$$
Let $h^{*}=\alpha^{j}\beta^{z}$, $0\leq j\leq mn-1,\ 1\leq z\leq n-1$, we have
$$\langle(z)!_{q}\sum_{k=0}^{mn-1}\xi^{jk}\overline{g^{k}x^{z}},\sum_{\substack{0\leq k\leq mn-1\\0\leq l\leq n-1}} b_{kl}g^{k}x^{l}\rangle\sum_{i=0}^{mn-1}\xi^{(n-1)i}\alpha^{i}\beta^{n-1}=0.$$
Since $\{\alpha^{i}\beta^{k}|0\leq i\leq mn-1, 0\leq j\leq n-1\}$ is {a} basis of $(R_{mn}(q))^{*},$ we have
$$\langle(z)!_{q}\sum_{k=0}^{mn-1}\xi^{jk}\overline{g^{k}x^{z}},\sum_{\substack{0\leq k\leq mn-1\\0\leq l\leq n-1}} b_{kl}g^{k}x^{l}\rangle=(z)!_{q}\sum_{k=0}^{mn-1}\xi^{jk}b_{kz}=0.$$
By Cramer's Rule and Vandermonde determinant, we have
$$b_{kl}=0,\ 0\leq k\leq mn-1,\ 1\leq l\leq n-1.$$
Therefore, $\tilde{g}=g^{1-n}.$
\end{proof}
Similarly, we can prove Part (2).

\begin{theorem}\label{4.17}
(1) $(D(R_{mn}(q)),\mathcal{R})$ has quasi-ribbon elements if and only if $n$ is odd.

(2) $(D(R_{mn}(q)),\mathcal{R})$ has {a} unique ribbon element if and only if both $m$ and $n$ are odd.

(3) $(D(R_{mn}(q)),\mathcal{R})$ has two ribbon elements if and only if $m$ is even {and} $n$ is odd.
\end{theorem}
\begin{proof}
(1) By part(1) of Theorem \ref{4.6}, $(D(R_{mn}(q)),\mathcal{R})$ has a quasi-ribbon element if and only if there {exist} $h=g^{j}\in G(R_{mn}(q))$, $\gamma=(\alpha^{-m})^{k}\in G(R_{mn}(q)^{*}),\ where\ j,\ k\in\mathbb{Z}$, such that $(g^{j})^{2}=g^{1-n},$ $(\alpha^{-m})^{2k}=\alpha^{-m}$, which implies $mn\mid {2j+n-1}$ and $n\mid {2k-1}.$ Since $2k-1$ is odd, $n$ must be odd. Furthermore, $2j+n-1$ is even. So no matter {whether} $m$ is even or odd, there exists $j=\frac{1-n}{2}\in\mathbb{Z}$ such that $mn\mid {2j+n-1}$. Thus, one knows that $(D(R_{mn}(q)),\mathcal{R})$ has quasi-ribbon elements if and only if $n$ is odd.

(2)
By part(2) of Theorem \ref{4.6}, $(D(R_{mn}(q)),\mathcal{R})$ has ribbon elements if and only if there {exist} $\gamma=(\alpha^{-m})^{k}\in G(R_{mn}(q)^{*})$, $h=g^{j}\in G(R_{mn}(q))$ which satisfy
\begin{equation}\label{gs4.4}
\begin{split}
&h^{2}=g^{1-n},\   \gamma^{2}=\alpha^{-m},\\
&S^{2}(x)=h(\gamma\rightharpoonup x\leftharpoonup\gamma^{-1})h^{-1},\\&S^{2}(g)=h(\gamma\rightharpoonup g\leftharpoonup\gamma^{-1})h^{-1},
\end{split}
\end{equation}
where $x$ and $g$ are the generators of $R_{mn}(q)$, $k,j\in \mathbb{Z}$ and $\gamma,h$ are given in part$(1)$.

It follows {from} $S^{2}(x)=q^{-1}x$, $h(\gamma\rightharpoonup x\leftharpoonup\gamma^{-1})h^{-1}=g^{j}(q^{-k}x)g^{-j}=q^{-1}x$, $S^{2}(g)=g$, $h(\gamma\rightharpoonup g\leftharpoonup\gamma^{-1})h^{-1}=g^{j}gg^{-j}=g$ that $(D(R_{mn}(q)),\mathcal{R})$ has ribbon elements if and only if there exists pairs $(\gamma,h)=((\alpha^{-m})^{k},g^{j})$ such that
\begin{equation*}
\begin{split}
&g^{2j}=g^{1-n},\   \alpha^{-2mk}=\alpha^{-m},\\
&q^{-1}x=q^{-k-j}x.
\end{split}
\end{equation*}
Since the order of $\alpha$ is $mn$ and the order of $g$ is $mn$, $(D(R_{mn}(q)),\mathcal{R})$ has ribbon elements if and only if there exists pairs $(\gamma,h)=((\alpha^{-m})^{k},g^{j})$ such that
$$mn\mid{2j+n-1},n\mid {2k-1},n\mid {k+j-1}.$$

$\Leftrightarrow$\begin{equation}\label{gs4.5}
j=\frac{mnp+1-n}{2}\in\mathbb{Z},k=\frac{nt+1}{2}\in\mathbb{Z},mp+t=2r+1,
\end{equation}
where $p,t,r\in \mathbb{Z}.$

If $t$ is odd, then $(\alpha^{-m})^{k}=(\alpha^{-m})^{\frac{nt+1}{2}}=(\alpha^{-m})^{\frac{n+1}{2}}$; if $t$ is even, $k=\frac{nt+1}{2}\notin \mathbb{Z}$, which implies that $t$ must be odd. If $p$ is odd, then $g^{j}=g^{\frac{mnp+1-n}{2}}=g^{\frac{mn+1-n}{2}}$; if $p$ is even, then $g^{j}=g^{\frac{mnp+1-n}{2}}=g^{\frac{1-n}{2}}$. By part$(1)$, one knows that if $(D(R_{mn}(q)),\mathcal{R})$ has ribbon elements, then $n$ must be odd. If both $n$ and $m$ are odd, $t$ must be odd since $k=\frac{nt+1}{2}\in\mathbb{Z}$. Moreover, $j=\frac{mnp+1-n}{2}\in\mathbb{Z}$ and $mp+t=2r+1$ imply that $p$ is even, then there exists unique pair $(\gamma,h)=((\alpha^{-m})^{\frac{n+1}{2}},g^{\frac{1-n}{2}})$ satisfying ($\ref{gs4.4}$). If $n$ is odd and $m$ is even, $t$ must be odd since $k=\frac{nt+1}{2}\in\mathbb{Z}$. In this case, no matter {whether} $p$ is odd or even, $j=\frac{mnp+1-n}{2}\in\mathbb{Z}$ and $mp+t=2r+1$. Thus, there exists two pairs $(\gamma_{1},h_{1})=((\alpha^{-m})^{\frac{n+1}{2}},g^{\frac{1-n}{2}})$ and $(\gamma_{2},h_{2})=((\alpha^{-m})^{\frac{n+1}{2}},g^{\frac{(m-1)n+1}{2}})$ satisfying ($\ref{gs4.4}$). Consequently, by Corollary \ref{4.7}, $(D(R_{mn}(q)),\mathcal{R})$ has unique ribbon element if and only if both $m$ and $n$ are odd; $(D(R_{mn}(q)),\mathcal{R})$ has two ribbon elements if and only if $m$ is even, $n$ is odd.
\end{proof}

\subsection{Computation of the ribbon elements of $D(R_{mn}(q))$}\selabel{4.2}
~~

Throughout this subsection, assume {that} $n$ is an odd integer. Notice that $\tilde{\alpha}=\alpha^{-m}$ and $\tilde{g}=g^{1-n}$ are the distinguished grouplike elements in $(R_{mn}(q))^{*}$ and $R_{mn}(q)$, respectively. By the description about the distinguished grouplike element of Drinfeld double of a finite-dimensional quasi-triangular Hopf algebra in \cite{DE1}, the distinguished grouplike element in $D(R_{mn}(q))$ is $\tilde{\alpha}\bowtie \tilde{g}$.

Recall the universal R-matrix of $D(R_{mn}(q))$ given in subsection 4.1:
\begin{equation}\label{4.2}
\begin{split}
\mathcal{R}&=\frac{1}{mn}\sum_{i,j,k}\frac{1}{(j)!_{q}}\xi^{-ik}(1\bowtie g^{i}x^{j})\otimes(\alpha^{k}\beta^{j}\bowtie 1).\\
\end{split}
\end{equation}

\begin{theorem}\label{4.22}

(1) When $m$ is odd, the unique ribbon element in $D(R_{mn}(q))$ is
$$v=u(\alpha^{\frac{m(n+1)}{2}}\bowtie g^{\frac{n-1}{2}}),$$
$$where\ u=\frac{1}{mn}\sum_{\substack{0\leq i,k\leq mn-1\\0\leq j\leq n-1}}(-1)^{j}\frac{1}{(j)!_{q}}\xi^{-(i+j)k-\frac{j(j-1)m}{2}}(\alpha^{-mj-k}\beta^{j}\bowtie g^{i}x^{j}).$$
(2) When $m$ is even, the ribbon elements in $D(R_{mn}(q))$ are
$$v_{1}=u(\alpha^{\frac{m(n+1)}{2}}\bowtie g^{\frac{n-1}{2}}),\ v_{2}=u(\alpha^{\frac{m(n+1)}{2}}\bowtie g^{\frac{n(m+1)-1}{2}}),$$
$$where\ u=\frac{1}{mn}\sum_{\substack{0\leq i,k\leq mn-1\\0\leq j\leq n-1}}(-1)^{j}\frac{1}{(j)!_{q}}\xi^{-(i+j)k-\frac{j(j-1)m}{2}}(\alpha^{-mj-k}\beta^{j}\bowtie g^{i}x^{j}).$$
\end{theorem}
\begin{proof} (1) We adopt the previous conventions and set $g_{\alpha_{D}}=g_{\varepsilon_{D}},$ (which holds as $D(R_{mn}(q))$ is unimodular). By (\ref{4.1}) and (\ref{4.2}), we have
\begin{equation*}
\begin{split}
g_{\varepsilon_{D}}&=\frac{1}{mn}\sum_{\substack{0\leq i,k\leq mn-1\\0\leq j\leq n-1}}\frac{1}{(j)!_{q}}\xi^{-ik}\varepsilon(\alpha^{k}\beta^{j}\bowtie1)(1\bowtie g^{i}x^{j})\\
&=\frac{1}{mn}\sum_{\substack{0\leq i,k\leq mn-1\\0\leq j\leq n-1}}\frac{1}{(j)!_{q}}\xi^{-ik}\varepsilon(1)\alpha^{k}\beta^{j}(1)(1\bowtie g^{i}x^{j}).\\
\end{split}
\end{equation*}
Since $\beta^{j}(1)=0$ when $j\neq0$ and $\varepsilon_{D}$ is an algebra homomorphism, only the terms with $j=0$ survive, and therefore
\begin{equation*}
\begin{split}
g_{\varepsilon_{D}}&=\frac{1}{mn}\sum_{\substack{0\leq i,k\leq mn-1\\0\leq j\leq n-1}}\xi^{-ik}(1\bowtie g^{i})\\
&=\frac{1}{mn}\sum_{i=0}^{mn-1}(\sum_{k=0}^{mn-1}\xi^{-ik})(1\bowtie g^{i}).\\
\end{split}
\end{equation*}

Observe that
$$\sum_{k=0}^{mn-1}\xi^{-ik}=\frac{1-(\xi^{-i})^{mn}}{1-\xi^{-i}}=0$$
unless $i=0$, in which case $\sum_{k=0}^{mn-1}\xi^{-ik}=mn$. Therefore $g_{\varepsilon_{D}}=1_{D(R_{mn}(q))}.$

By the discussion above, the distinguished grouplike element in $D(R_{mn}(q))\simeq (R_{mn}(q))^{*}\otimes R_{mn}(q)$ is $\hat{g}=\alpha^{-m}\bowtie g^{1-n}$.

By (\ref{4.1}), $h_{\varepsilon_{D}}=g_{\varepsilon_{D}}(\hat{g})^{-1}=(\alpha^{-m}\bowtie g^{1-n})^{-1}=\alpha^{m}\bowtie g^{n-1}.$ When $m$ and $n$ are both odd, the square root $h_{\varepsilon_{D}}^{'}$ of $h_{\varepsilon_{D}}$ is unique, because $h_{\varepsilon_{D}}$, and therefore $h_{\varepsilon_{D}}^{'}$, has odd order. Thus,
$$h_{\varepsilon_{D}}^{'}=\alpha^{\frac{m(n+1)}{2}}\bowtie g^{\frac{n-1}{2}},$$
$$v=uh_{\varepsilon_{D}}^{'}=u(\alpha^{\frac{m(n+1)}{2}}\bowtie g^{\frac{n-1}{2}}).$$
By Theorem \ref{4.5}, the quasi-ribbon element $v$ is the unique ribbon element of $D(R_{mn}(q))$.

(2) When $m$ is even and $n$ is odd, $h_{\varepsilon_{D}}$ has four square roots
$$h_{\varepsilon_{D1}}^{'}=\alpha^{\frac{m(n+1)}{2}}\bowtie g^{\frac{n-1}{2}},\ h_{\varepsilon_{D2}}^{'}=\alpha^{\frac{m}{2}}\bowtie g^{\frac{n-1}{2}},$$
$$h_{\varepsilon_{D3}}^{'}=\alpha^{\frac{m(n+1)}{2}}\bowtie g^{\frac{n(m+1)-1}{2}},\ h_{\varepsilon_{D4}}^{'}=\alpha^{\frac{m}{2}}\bowtie g^{\frac{n(m+1)-1}{2}}.$$
By Theorem \ref{4.5}, $$S^{2}(\varepsilon\bowtie g)=(h_{\varepsilon_{Di}}^{'})^{-1}(\varepsilon\bowtie g)h_{\varepsilon_{Di}}^{'},\ S^{2}(\varepsilon\bowtie x)=(h_{\varepsilon_{Di}}^{'})^{-1}(\varepsilon\bowtie x)h_{\varepsilon_{Di}}^{'},$$ $$S^{2}(\alpha\bowtie1)=(h_{\varepsilon_{Di}}^{'})^{-1}(\alpha\bowtie1)h_{\varepsilon_{Di}}^{'},\ S^{2}(\beta\bowtie1)=(h_{\varepsilon_{Di}}^{'})^{-1}(\beta\bowtie1)h_{\varepsilon_{Di}}^{'},$$ where $i=1,3$ and $\varepsilon\bowtie g,\varepsilon\bowtie x,\alpha\bowtie1,\beta\bowtie1$ are the generators of $D(R_{mn}(q))$.

Therefore, quasi-ribbon elements $$v_{1}=uh_{\varepsilon_{D1}}^{'}=u(\alpha^{\frac{m(n+1)}{2}}\bowtie g^{\frac{n-1}{2}})\ and\  v_{2}=uh_{\varepsilon_{D3}}^{'}=u(\alpha^{\frac{m(n+1)}{2}}\bowtie g^{\frac{n(m+1)-1}{2}})$$ are the ribbon elements of $D(R_{mn}(q))$.

It remains to show that $u$ has the expression in (\ref{2}). Recall that $u=\sum_{i}S(y_{i})x_{i}$, where $$\mathcal{R}=\frac{1}{mn}\sum_{i,j,k}\frac{1}{(j)!_{q}}\xi^{-ik}(1\bowtie g^{i}x^{j})\otimes(\alpha^{k}\beta^{j}\bowtie 1).$$ Therefore,
\begin{equation*}
\begin{split}
u&=\frac{1}{mn}\sum_{\substack{0\leq i,k\leq mn-1\\0\leq j\leq n-1}}\frac{1}{(j)!_{q}}\xi^{-ik}S(\alpha^{k}\beta^{j}\bowtie 1)(1\bowtie g^{i}x^{j})\\
&=\frac{1}{mn}\sum_{\substack{0\leq i,k\leq mn-1\\0\leq j\leq n-1}}\frac{1}{(j)!_{q}}\xi^{-ik}((-\alpha^{-m}\beta)^{j}(\alpha^{mn-1})^{k})\bowtie1)(1\bowtie g^{i}x^{j})\\
&=\frac{1}{mn}\sum_{\substack{0\leq i,k\leq mn-1\\0\leq j\leq n-1}}(-1)^{j}\frac{1}{(j)!_{q}}\xi^{-ik-jk-\frac{j(j-1)m}{2}}(\alpha^{-mj-k}\beta^{j}\bowtie1)(1\bowtie g^{i}x^{j})\\
&=\frac{1}{mn}\sum_{\substack{0\leq i,k\leq mn-1\\0\leq j\leq n-1}}(-1)^{j}\frac{1}{(j)!_{q}}\xi^{-(i+j)k-\frac{j(j-1)m}{2}}(\alpha^{-mj-k}\beta^{j}\bowtie g^{i}x^{j}).
\end{split}
\end{equation*}
\end{proof}

\end{document}